\pdfoutput=1
\documentclass{louloupart1}
\usepackage{mathtools}
\usepackage{graphicx}
\usepackage{hyperref}
\usepackage{bm}
\usepackage{tikz}
\usetikzlibrary{shapes.geometric,calc}
\usetikzlibrary{decorations.markings,arrows,arrows.meta,positioning,cd}
\tikzset{->-/.style={decoration={
  markings,
  mark=at position #1 with {\arrow{Computer Modern Rightarrow[length=5pt,width=5pt]}}},postaction={decorate}}}
\tikzset{->-rev/.style={decoration={
  markings,
  mark=at position #1 with {\arrow{Computer Modern Rightarrow[length=5pt,width=5pt,reversed]}}},postaction={decorate}}}

\title[Property $R_\infty$ for Artin and Coxeter groups whose outer automorphisms are length-preserving]
{Property $R_\infty$ for Artin and Coxeter groups whose outer\\ automorphisms are length-preserving}
\author[I Soroko]{Ignat Soroko}
\givenname{Ignat}
\surname{Soroko}
\address{Ignat Soroko, Department of Mathematics, Southern Methodist University, Dallas, TX 75205, USA}
\email{isoroko@smu.edu}
\date{\today}

\newtheorem{thm}{Theorem} 
\newtheorem{lem}[thm]{Lemma}

\newtheorem{corl}[thm]{Corollary}

\newtheorem{quest}{Question}

\theoremstyle{definition}

\newtheorem{rem}[thm]{Remark}

\newtheorem*{expl}{Example}
\newtheorem*{acknow}{Acknowledgments}

\numberwithin{equation}{section}

\begin{document}

\def\N{\mathbb N} \def\Inn{{\rm Inn}} \def\Out{{\rm Out}} \def\Z{\mathbb Z}
\def\id{{\rm id}} \def\supp{{\rm supp}} 
\renewcommand{\Im}{\operatorname{Im}} 
\def\PP{\mathcal P} \def\Homeo{{\rm Homeo}}
\def\SHomeo{{\rm SHomeo}} \def\LHomeo{{\rm LHomeo}}
\def\MM{\mathcal M} \def\CC{\mathcal C} \def\AA{\mathcal A}
\def\S{\mathbb S} \def\FF{\mathcal F} \def\SS{\mathcal S}
\def\LL{\mathcal L} \def\D{\mathbb D} 

\newcommand{\Ker}{\operatorname{Ker}}
\newcommand{\Cyl}{\operatorname{Cyl}}
\newcommand{\Aut}{\operatorname{Aut}}
\newcommand{\conj}{\operatorname{conj}}
\newcommand{\card}{\operatorname{Card}}
\newcommand{\Sym}{\operatorname{Sym}}
\newcommand{\Cent}{\operatorname{Cent}}
\newcommand{\Fix}{\operatorname{Fix}}
\newcommand{\Isom}{\operatorname{Isom}}
\newcommand{\ov}{\overline}
\mathchardef\mhyphen="2D
\renewcommand{\le}{\leqslant}
\renewcommand{\ge}{\geqslant}
\newcommand{\lk}{\operatorname{lk}}

\begin{abstract}
We prove that every length-preserving automorphism of a nonspherical Artin or Coxeter group has infinite Reidemeister number. As a consequence, we establish property $R_\infty$ for large-type free-of-infinity Artin groups, infinite strongly rigid Coxeter groups, and infinite virtually abelian Coxeter groups, including all affine Coxeter groups. We also prove that the automorphism group of every infinite irreducible non-affine strongly rigid Coxeter group is acylindrically hyperbolic.

\smallskip\noindent
{\bf 2020 Mathematics Subject Classification~} 
Primary: 20F36, 20F55; Secondary: 20E36, 20E45.

\smallskip\noindent
{\bf Keywords\ \ } 
Artin groups, Coxeter groups, property $R_\infty$, twisted conjugacy, acylindrical hyperbolicity.

\end{abstract}

\maketitle


\section{Introduction}\label{sec1}

Let $G$ be a group and let $\varphi$ be an automorphism of $G$. We say that elements $g,h\in G$ are \emph{$\varphi$-twisted conjugate} if there exists $x\in G$ such that $h=x\,g\, \varphi(x)^{-1}$. This defines an equivalence relation on $G$, and the number of its equivalence classes (possibly, infinite) is called the \emph{Reidemeister number} of $\varphi$, denoted by $R(\varphi)$. We say that~$G$ \emph{has property $R_{\infty}$} if $R(\varphi)=\infty$ for all $\varphi\in \Aut(G)$. 

Property $R_\infty$ is an interesting and rather mysterious property of groups, which has its origins in the Nielsen fixed point theory, and the problem of determining which classes of groups possess it remains an area of active research. The list of groups known to have property $R_\infty$ is quite large and contains non-elementary hyperbolic groups, broad classes of relatively hyperbolic groups and all finitely generated groups with infinitely many ends~\cite{Felsh1,FFW23,FFIMSW26}. Other prominent examples include mapping class groups, higher rank lattices and certain solvable, nilpotent, and branch groups~\cite{FelTro2,FelNas1}. In particular, the free nilpotent group $N_{r,c}$ of rank $r\ge2$ and nilpotency class $c$ has property $R_\infty$ if and only if $c\ge 2r$~\cite{DekGon1}. It was conjectured that every finitely generated residually finite nonamenable group has property $R_\infty$~\cite{FelTro2,SorVas1}. For some other recent developments, see~\cite{DekLat1,Troit1,DeLiRe1,SorVas2}.

Let $S$ be a finite set. A \emph{Coxeter matrix} over $S$ is a symmetric matrix $(m_{st})_{s,t\in S}$ with entries in $\{1,2,\dots\}\cup\{\infty\}$, such that $m_{ss}=1$ for all $s\in S$ and $m_{st}=m_{ts}\ge2$ if $s\ne t$. 
A Coxeter matrix can be encoded by the corresponding \emph{Coxeter graph} $\Gamma$ with vertex set $S$: two distinct vertices $s,t\in S$ are joined by an edge precisely when $m_{st}\ne2$, and the edge is labeled by $m_{st}$, with the label $3$ usually omitted (thus an edge may carry the label $\infty$). The Coxeter system is \emph{irreducible} if $\Gamma$ is connected. 
The \emph{Artin group of type $\Gamma$} is the group $A[\Gamma]$ given by the presentation
\[
A[\Gamma]=\langle S\mid \Pi(s,t,m_{st})=\Pi(t,s,m_{ts}), \text{ for all } s\ne t,\,m_{st}\ne\infty\rangle,
\]
where $\Pi(s,t,m_{st})$ is the word $stst\dots$ of length $m_{st}\ge2$. The \emph{Coxeter group $W[\Gamma]$ of type $\Gamma$} is the quotient of $A[\Gamma]$ by all relations of the form $s^2=1$, $s\in S$. We denote by $P[\Gamma]$ the kernel of the natural epimorphism $A[\Gamma]\to W[\Gamma]$; it is called the \emph{pure Artin group of type $\Gamma$}. An Artin group $A[\Gamma]$ and its Coxeter group $W[\Gamma]$ are called \emph{spherical} if $W[\Gamma]$ is finite, and \emph{nonspherical} if $W[\Gamma]$ is infinite. They are called \emph{right-angled} if $m_{st}\in\{2,\infty\}$ for all $s\ne t$. An Artin group $A[\Gamma]$ is called \emph{large} (or \emph{large-type}) if all $m_{st}\ge3$ for distinct $s,t\in S$ (with $m_{st}=\infty$ allowed). It is called \emph{hyperbolic-type} if its Coxeter group $W[\Gamma]$ is word-hyperbolic. (Notice that Artin groups of hyperbolic-type are not word-hyperbolic themselves, unless they are free.) An Artin group $A[\Gamma]$ is called \emph{free-of-infinity} if $m_{st}\ne\infty$ for all distinct $s,t\in S$. Coxeter groups of free-of-infinity Artin groups are also called \emph{$2$-spherical}. A Coxeter group $W$ is called \emph{strongly rigid} if any two Coxeter generating sets $S,S'$ for $W$ are conjugate.

Let $G$ denote an Artin or Coxeter group over $S$. We call an automorphism $\varphi$ of $G$ \emph{length-preserving} if $\varphi(s)\in S\cup S^{-1}$ for all $s\in S$. We call the automorphism $\iota$ of an Artin group defined on the generators $s\in S$ by $\iota(s)=s^{-1}$ the \emph{global inversion of generators}. Notice that $\iota$ is never inner, since inner automorphisms act as the identity on the abelianization, whereas $\iota$ acts as $-\id$.

For Artin groups, property $R_\infty$ was established for the following types:
\begin{itemize}
\item braid groups (i.e.\ Artin groups of type~$A_n$) in~\cite{FeGoDa1};
\item pure braid groups $P[A_n]$ in~\cite{DeGoOc1};
\item spherical groups of types $A_n$, $B_n$, $D_4$, $I_2(m)$ for $m\ge3$, and their pure subgroups in~\cite{CalSor1};
\item spherical groups of type $D_n$, $n\ge5$, in~\cite{SorVas1};
\item affine groups of types $\widetilde A_n$, $\widetilde C_n$ in~\cite{CalSor1};
\item nonabelian right-angled Artin groups in~\cite{Witdo1} (see also~\cite{DekSen1});
\item large hyperbolic-type free-of-infinity Artin groups (and some other classes of large hyperbolic-type Artin groups) in~\cite{SorVas1}.
\end{itemize}

For Coxeter groups, property $R_\infty$ was established for the following types:
\begin{itemize}
\item non-elementary Gromov-hyperbolic Coxeter groups, by the general theorem for hyperbolic groups~\cite{Felsh1}. A criterion for hyperbolicity of Coxeter groups is given in~\cite[Corollary~12.6.3]{Davis1};

\item suitable relatively hyperbolic Coxeter groups, by the general theorem for relatively hyperbolic groups~\cite[Theorem~6.1.2]{FFIMSW26}. Criteria for relative hyperbolicity of Coxeter groups are given in~\cite{Capra1,Capra2};

\item Coxeter groups with infinitely many ends, by the general theorem for groups with infinitely many ends~\cite[Corollary~8.1.6(3)]{FFIMSW26}. A criterion for determining when a Coxeter group has infinitely many ends is given in~\cite[Theorem~8.7.4]{Davis1};

\item non-virtually-abelian right-angled Coxeter groups. Indeed, these are graph products of copies of $\mathbb Z/2\mathbb Z$, and are covered by~\cite[Corollary~8.1.6(5)]{FFIMSW26};

\item odd Coxeter groups of rank at least $3$ (i.e.\ all $m_{st}$ are odd or $\infty$) for which the graph on $S$ joining $s,t$ whenever $m_{st}<\infty$ is a tree~\cite{NaiSin1}. These groups are virtually free, hence are non-elementary Gromov-hyperbolic, 
thus their $R_\infty$ property is also subsumed by the first item above;

\item twin groups~\cite{NaNaSi1}, i.e.\ right-angled Coxeter groups whose Coxeter graph is a path with every edge labeled $\infty$. For $n\ge4$ the twin group $T_n$ (of rank $n-1$) is non-virtually-abelian, so this result is now also subsumed by the preceding right-angled Coxeter group case, whereas $T_3$ is the infinite dihedral group $D_\infty$. 
\end{itemize}

In this paper we give a short combinatorial proof of the following theorem.

\begin{thm}\label{thm:1}
Let $G$ be a nonspherical Artin or Coxeter group, and let $\varphi$ be a length-preserving automorphism of $G$. Then $R(\varphi)=\infty$.
\end{thm}


The statement about Artin groups in Theorem~\ref{thm:1} strengthens earlier work of Juh\'asz~\cite{Juhas1} on extra-large-type Artin groups and on large-type Artin groups for which the defining graph, i.e.\ the graph joining $s,t$ whenever $m_{st}<\infty$, is connected and triangle-free.

\begin{corl}\label{cor:outer-lp}
Let $G$ be a nonspherical Artin or Coxeter group. If every class in $\Out(G)$ has a length-preserving representative, then $G$ has property $R_\infty$. In particular, this holds if $\Out(G)$ is generated by automorphisms of the Coxeter graph and, in the Artin case, the global inversion.
\end{corl}

Known descriptions of automorphism groups allow us to apply Corollary~\ref{cor:outer-lp} to several classes of Artin and Coxeter groups, listed below.

\begin{corl}\label{cor:2}
Let $A[\Gamma]$ be a large-type free-of-infinity Artin group of rank at least $2$. Then $A[\Gamma]$ has property $R_\infty$.
\end{corl}

Corollary~\ref{cor:2} strengthens the result of~\cite{SorVas1}, where the same conclusion was proved under the additional assumption that $A[\Gamma]$ is of hyperbolic type.

\begin{corl}\label{cor:coxeter-classes}
The following classes of Coxeter groups have property $R_\infty$:
\begin{enumerate}
\item every infinite Coxeter group whose infinite irreducible components are strongly rigid (with arbitrary finite components). In particular, this includes every infinite irreducible $2$-spherical Coxeter group;
\item every infinite virtually abelian Coxeter group. Equivalently, these are the Coxeter groups whose Coxeter graph is a disjoint union of spherical and affine components, with at least one affine component;
\end{enumerate}
\end{corl}

\begin{expl}\label{expl:finite-factor-example}
Let $W=D_\infty\times C_2=\langle a,b\mid a^2=b^2=1\rangle\times\langle z\mid z^2=1\rangle$. Then the automorphism $\tau$ defined by
$\tau(a)=az,\,\tau(b)=b,\,\tau(z)=z$ is not inner-by-graph: it sends the reflection $a$ to the non-reflection $az$. Thus, $W$ itself is not strongly rigid, but its infinite dihedral component $D_\infty$ is (see~\cite[Lemma~7.1.1]{FFIMSW26}). Hence, Corollary~\ref{cor:coxeter-classes}(1) establishes $R_\infty$ for $W$.
\end{expl}

A separate geometric argument yields the following result on automorphism groups.

\begin{thm}\label{thm:aut-ah}
Let $(W,S)$ be an infinite irreducible non-affine strongly rigid Coxeter system. Then $\Aut(W)$ is acylindrically hyperbolic. In particular, this applies to every infinite irreducible $2$-spherical non-affine Coxeter system.
\end{thm}

Theorem~\ref{thm:aut-ah} has many remarkable consequences, see~\cite[Section~3]{FFW23}. In particular, it implies that $W$ is acylindrically hyperbolic itself and admits an infinite simple $2$-generated characteristic quotient (see~\cite{CouFou1}). 

\begin{acknow}
The author is grateful to Francesco Fournier-Facio, Kaitlin Ragosta and Nicolas Vaskou for their interest to this topic and useful discussions.
\end{acknow}

\section{Preliminaries}\label{sec:twisted}

We first record two elementary facts. If $\varphi$ is length-preserving, then $\varphi(S\cup S^{-1})\subseteq S\cup S^{-1}$. Since $\varphi$ is injective and $S\cup S^{-1}$ is finite, it permutes this set. In particular, $\varphi$ has finite order.

\begin{lem}[\protect{\cite[Proposition~4]{CalSor1}}]\label{lem:cosets}
Let $G$ be a group, let $\varphi\in\Aut(G)$ have finite order $m$, and set
\[
G_\varphi=G\rtimes_\varphi\langle t\mid t^m=1\rangle, \qquad tgt^{-1}=\varphi(g)\quad\text{for all }g\in G.
\]
Then $g,h\in G$ are $\varphi$-twisted conjugate if and only if $gt$ and $ht$ are conjugate in $G_\varphi$. Moreover, whenever they are conjugate, they are conjugate by an element of $G$. Consequently, $R(\varphi)$ equals the number of conjugacy classes of $G_\varphi$ contained in $Gt$.
\end{lem}

We note that the assumption that $G$ is centerless, made in~\cite{CalSor1}, is unnecessary. The stronger assertion about conjugation by $G$
follows from the identity $t(gt)t^{-1}=g^{-1}(gt)g$. 

\begin{lem}\label{lem:quotient}
Let $q\colon G\to Q$ be an epimorphism and let $\varphi\in\Aut(G)$ preserve $\Ker q$, inducing $\bar\varphi\in\Aut(Q)$. Then $R(\varphi)\ge R(\bar\varphi)$.
\end{lem}

\begin{proof}
The map $q$ sends $\varphi$-twisted conjugate elements to $\bar\varphi$-twisted conjugate elements, and hence induces a surjection on twisted conjugacy classes.
\end{proof}

Let $(W,S)$ be a Coxeter system with $S$ finite, let $\delta$ be a Coxeter graph automorphism of order $m$, and put
\[
\widehat W=W\rtimes\langle\delta\rangle.
\]
The graph automorphism $\delta$ acts naturally on the Davis complex $\Sigma$, so $\widehat W$ acts on $\Sigma$ by isometries.

Let 
\[
\delta\backslash S=\{\langle\delta\rangle s\mid s\in S\}
\] 
denote the set of $\langle\delta\rangle$-orbits in $S$. 

For $I\in\delta\backslash S$, let
\[
W_I=\langle I\rangle
\]
be the standard parabolic subgroup generated by $I$; we call these the \emph{$\delta$-orbit parabolic subgroups}. If $W_I$ is finite, denote its longest element by $w_0(I)$.

We shall use two results on the fixed subgroup $W^\delta=\{w\in W\mid \delta(w)=w\}$.

The first is an extension of a theorem of Steinberg~\cite[Theorem~32]{Ste68} to infinite Coxeter groups.
\begin{thm}[\protect{H\'ee--M\"uhlherr \cite[Corollaire~3.5]{Hee91}, \cite[Theorem~1.3]{Muh93}}]\label{thm:hee-fixed}
The fixed subgroup $W^\delta$ is a Coxeter group with Coxeter generating set $S(W^\delta)=\{w_0(I)\mid I\in\delta\backslash S,\; W_I\text{ finite}\}$.
\end{thm}

\begin{thm}[\protect{Nuida~\cite[Theorem~5.2(1)]{Nui07}}]\label{thm:nuida-fixed}
Suppose that $(W,S)$ is irreducible and nonspherical. If every $\delta$-orbit parabolic subgroup $W_I$, $I\in\delta\backslash S$, is finite, then the Coxeter system $(W^\delta,S(W^\delta))$ is irreducible and nonspherical.
\end{thm}

\section{Proofs}
Let, as above, $(W,S)$ be a Coxeter system with $S$ finite, let $\delta$ be a Coxeter graph automorphism of order~$m$, and let $W\delta$ denote a coset of $\widehat W=W\rtimes\langle \delta\rangle$.

\begin{lem}\label{lem:one-infinite-order}
If $W\delta$ contains an element of infinite order, then $R(\delta)=\infty$.
\end{lem}

\begin{proof}
Let $g\in W\delta$ have infinite order and set $h=g^m\in W$. Then $h$ also has infinite order. For $w\in W$, let $|w|_\Sigma$ denote its 
translation length on the Davis complex $\Sigma$, i.e.\ $|w|_\Sigma=\inf_{x\in\Sigma}d_\Sigma(x,wx)$.

The group $W$ acts properly and cocompactly on $\Sigma$, which is CAT$(0)$ (see~\cite[Chapter~12]{Davis1}), so every element of $W$ acts semisimply by~\cite[Proposition~II.6.10(2)]{BH99}. Since point stabilizers are finite, $h$ cannot be elliptic and is therefore hyperbolic. By~\cite[Proposition~II.6.8(1)]{BH99}, $h$ admits an axis, and hence $|h|_\Sigma>0$ and $|h^n|_\Sigma=n|h|_\Sigma$ for all $n\ge1$.

For $k\ge0$ set $n_k=1+km$. Then $g^{n_k}\in W\delta$. If $g^{n_i}$ and $g^{n_j}$ were conjugate in $\widehat W$, their $m$-th powers $h^{n_i}$ and $h^{n_j}$ would also be conjugate. Since $\widehat W$ acts on $\Sigma$ by isometries and translation length is conjugacy invariant,
\[
n_i|h|_\Sigma=|h^{n_i}|_\Sigma=|h^{n_j}|_\Sigma=n_j|h|_\Sigma\,,
\]
so $i=j$. Thus, $W\delta$ contains infinitely many conjugacy classes, and Lemma~\ref{lem:cosets} gives $R(\delta)=\infty$.
\end{proof}

\begin{lem}\label{lem:coset-infinite}
Let $(W,S)$ be a nonspherical Coxeter system with $S$ finite, and let $\delta$ be a Coxeter graph automorphism. Then the coset $W\delta$ contains an element of infinite order.
\end{lem}

\begin{proof}
Assume first that $(W,S)$ is irreducible, and let $m$ be the order of $\delta$. We distinguish two cases according to the finite collection of $\delta$-orbit parabolic subgroups $W_I$, $I\in\delta\backslash S$.

Suppose first that $W_I$ is infinite for some $I\in\delta\backslash S$. Choose $s\in I$ and put $d=|I|$. Since $I$ is the $\langle\delta\rangle$-orbit of $s$,
\[
I=\{s,\delta(s),\ldots,\delta^{d-1}(s)\},
\qquad d\mid m.
\]
Thus
\[
c_I=s\,\delta(s)\cdots\delta^{d-1}(s)
\]
is a Coxeter element of the Coxeter system $(W_I,I)$, since every generator in $I$ occurs exactly once. In the semidirect product $W\rtimes\langle\delta\rangle$ we have
\[
(s\delta)^m =s\,\delta(s)\cdots\delta^{m-1}(s) =\bigl(s\,\delta(s)\cdots\delta^{d-1}(s)\bigr)^{m/d} =c_I^{m/d}.
\]
Since $W_I$ is infinite, its Coxeter element $c_I$ has infinite order by Howlett~\cite[Theorem~4.1]{How82}. Hence $(s\delta)^m$, and therefore $s\delta$, has infinite order.

Suppose now that every $\delta$-orbit parabolic subgroup $W_I$ is finite. By Theorems~\ref{thm:hee-fixed} and~\ref{thm:nuida-fixed},
\[
\bigl(W^\delta,\{w_0(I)\mid I\in\delta\backslash S\}\bigr)
\]
is an irreducible nonspherical Coxeter system. Let $c$ be a Coxeter element of this system. By Howlett~\cite[Theorem~4.1]{How82}, $c$ has infinite order. Since $c\in W^\delta$, we have $\delta(c)=c$, and hence
\[
(c\delta)^m=c\,\delta(c)\cdots\delta^{m-1}(c)=c^m.
\]
Therefore $c\delta$ has infinite order.

Finally, suppose that $(W,S)$ is reducible. Write $S$ as the disjoint union of the irreducible components of the Coxeter system. Since $W$ is nonspherical, there is an irreducible component $S_0\subseteq S$ such that $W_{S_0}$ is infinite. Let $r$ be the size of the $\langle\delta\rangle$-orbit of $S_0$. Thus
\[
\mathcal O=\{S_0,\delta(S_0),\ldots,\delta^{r-1}(S_0)\}
\]
are pairwise distinct irreducible components and $\delta^r(S_0)=S_0$. Denote $\delta_0=\delta^r|_{W_{S_0}}$. 

The Coxeter system $(W_{S_0},S_0)$ is irreducible and nonspherical, so the irreducible case gives an element $u\in W_{S_0}$ such that $u\delta_0$ has infinite order in $W_{S_0}\rtimes\langle\delta_0\rangle$. Let 
\[
W_{\mathcal O} =W_{S_0}\times W_{\delta(S_0)}\times\cdots \times W_{\delta^{r-1}(S_0)}.
\]
Since the components in $\mathcal O$ are distinct, their corresponding parabolic subgroups form a direct product. Let $p_0\colon W_{\mathcal O}\to W_{S_0}$ be the projection onto the first factor. Since $\delta$ permutes the $r$ components in $\mathcal O$ cyclically, we have $r\mid m$.

Consider $g=u\delta$. Taking its $m$-th power gives
\[
g^m=u\,\delta(u)\cdots\delta^{m-1}(u)\in W_{\mathcal O}.
\]
Each factor $\delta^j(u)$ belongs to $W_{\delta^j(S_0)}$. Thus, projecting onto $W_{S_0}$ eliminates all factors except those for which $r$ divides $j$. Consequently,
\[
p_0(g^m)=u\,\delta^r(u)\,\delta^{2r}(u)\cdots\delta^{m-r}(u) =u\,\delta_0(u)\,\delta_0^2(u)\cdots\delta_0^{m/r-1}(u) =(u\delta_0)^{m/r},
\]
where the last equality holds because $\delta_0^{m/r}=1$. Since $u\delta_0$ has infinite order by construction, so does its power $(u\delta_0)^{m/r}$. Therefore $g^m$, and hence $g=u\delta$, has infinite order.
\end{proof}

\begin{proof}[Proof of Theorem~\ref{thm:1}]
First suppose that $G=W[\Gamma]$ is a Coxeter group. Since $s^{-1}=s$ for $s\in S$, the automorphism $\varphi$ permutes $S$, and hence is a Coxeter graph automorphism. Lemmas~\ref{lem:coset-infinite} and~\ref{lem:one-infinite-order} give $R(\varphi)=\infty$.

Now suppose that $G=A[\Gamma]$ is an Artin group, and let $q:A[\Gamma]\to W[\Gamma]$ be the natural quotient. Since $\varphi$ permutes $S\cup S^{-1}$, it preserves the normal closure of $\{s^2:s\in S\}$ and therefore induces an automorphism $\bar\varphi$ of $W[\Gamma]$. The automorphism $\bar\varphi$ permutes $S$, hence is a Coxeter graph automorphism. Since $W[\Gamma]$ is infinite, the Coxeter case gives $R(\bar\varphi)=\infty$, and Lemma~\ref{lem:quotient} gives $R(\varphi)=\infty$.
\end{proof}

\begin{proof}[Proof of Corollary~\ref{cor:outer-lp}]
By~\cite[Lemma~2.1]{FeGoDa1}, the Reidemeister number depends only on the class of an automorphism in $\Out(G)$. By assumption, every such class has a length-preserving representative, whose Reidemeister number is infinite by Theorem~\ref{thm:1}. Hence $G$ has property $R_\infty$.
\end{proof}

\begin{proof}[Proof of Corollary~\ref{cor:2}]
If $A[\Gamma]$ has rank $2$, it is a dihedral Artin group, and the result is~\cite[Theorem~1]{CalSor1}. Assume that the rank is at least $3$. By Vaskou~\cite[Theorem~A]{Vas25}, its outer automorphism group is generated by graph automorphisms and the global inversion. Since a large-type free-of-infinity Coxeter group of rank at least $3$ is nonspherical, Corollary~\ref{cor:outer-lp} applies.
\end{proof}

\begin{proof}[Proof of Corollary~\ref{cor:coxeter-classes}]
For (1), write
\[
W=W_\infty\times W_f,
\]
where $W_\infty$ is the product of the infinite irreducible components and $W_f$ is the product of the finite ones. By~\cite[Proposition~4.3]{Par07}, an infinite irreducible Coxeter group has no nontrivial finite normal subgroup. Hence $W_f$ is the largest finite normal subgroup of $W$, and therefore is characteristic. Every automorphism of $W$ consequently induces an automorphism of $W/W_f\cong W_\infty$.

The group $W_\infty$ is centerless, and by~\cite[Proposition~17.3]{HMN18} a direct product of strongly rigid irreducible Coxeter groups with trivial center is strongly rigid. Thus every outer automorphism class of $W_\infty$ contains a graph automorphism, and Corollary~\ref{cor:outer-lp} gives property $R_\infty$ for $W_\infty$. Lemma~\ref{lem:quotient} now gives the result for $W$. Infinite irreducible $2$-spherical Coxeter systems are strongly rigid by~\cite[Theorem~1]{FHM06}.

For (2), an irreducible Coxeter group is virtually abelian exactly when it is finite or affine~\cite[Theorem~12.3.5]{Davis1}. Since the connected components of the Coxeter graph give the irreducible direct product decomposition, $W$ is virtually abelian exactly when every connected component is spherical or affine; it is infinite exactly when at least one component is affine. Affine Coxeter systems are strongly rigid by~\cite[Proposition~17.7]{HMN18}, so (2) follows from (1).
\end{proof}

\section{Automorphism groups and acylindrical hyperbolicity}\label{sec:aut-ah}

Fournier-Facio asked us for which Coxeter groups $W$ the automorphism group $\Aut(W)$ is acylindrically hyperbolic. Theorem~\ref{thm:aut-ah} gives a partial answer to this question.

\begin{proof}[Proof of Theorem~\ref{thm:aut-ah}] We generalize the proof of~\cite[Lemma\,7.8]{AMS16} to the finite extension of $W$ by graph automorphisms. Let $D$ be the finite group of Coxeter graph automorphisms. Since $(W,S)$ is strongly rigid, every automorphism of $W$ is the composition of an inner automorphism and an element of $D$. Hence the natural homomorphism
\[
\Phi\colon W\rtimes D\longrightarrow\Aut(W),
\qquad
(w,\delta)\longmapsto\conj_w\circ\delta,
\]
is surjective.

The group $W\rtimes D$ acts properly by isometries on the Davis complex $\Sigma$. By Caprace--Fujiwara~\cite[Corollary~4.3 and Proposition~4.5]{CF10}, a Coxeter element of the irreducible non-affine group $W$ is rank one on $\Sigma$. Since $W\rtimes D$ is not virtually cyclic, it is
acylindrically hyperbolic by~\cite[Theorems~1.3 and~1.5]{Sis18}.

The kernel of this homomorphism is finite. Indeed, suppose that $(w,\delta)\in W\rtimes D$ lies in the kernel. Then $\conj_w\circ\delta=\id$, and hence $\delta=\conj_{w^{-1}}$. Thus, once $\delta$ is fixed, the element $w$ is uniquely determined: if both $(w,\delta)$ and $(w',\delta)$ lie in the kernel, then $\conj_{w^{-1}}=\delta=\conj_{w'^{-1}}$, so $w'w^{-1}\in Z(W)$. Since $W$ is infinite and irreducible, $Z(W)=1$ by~\cite[Theorem~D.2.10]{Davis1}, and therefore $w=w'$. Hence the projection
\[
\Ker\Phi\longrightarrow D,\qquad (w,\delta)\longmapsto\delta,
\]
is injective. Since $D$ is finite, $\Ker\Phi$ is finite. Acylindrical hyperbolicity passes to quotients by finite normal subgroups~\cite[Lemma~1]{MO19}, so $\Aut(W)$ is acylindrically hyperbolic.

The final assertion follows from~\cite[Theorem~1]{FHM06}: infinite irreducible $2$-spherical Coxeter systems are strongly rigid.
\end{proof}

\begin{rem}
By~\cite[Theorem~8.1.5]{FFIMSW26}, if $\Inn(G)$ is infinite and $\Aut(G)$ is acylindrically hyperbolic, then $G$ has property $R_\infty$. Hence Theorem~\ref{thm:aut-ah} gives an alternative proof of property $R_\infty$ for infinite irreducible $2$-spherical non-affine Coxeter groups. It does not, however, cover several classes for which $R_\infty$ was established by our combinatorial argument:
\begin{enumerate}
\item affine Coxeter groups, and more generally infinite virtually abelian Coxeter groups;
\item reducible Coxeter groups with at least two infinite irreducible components;
\item the Artin group applications, including large-type free-of-infinity Artin groups.
\end{enumerate}
\end{rem}

\section{Open questions}\label{sec:questions}
\subsection{Pure subgroups of Artin groups}
Franco--Paris~\cite{FP07} proved that the pure Artin subgroup $P[\Gamma]$ is characteristic for every Artin group of spherical type. This does not hold in general: for a universal Artin group $A[\Gamma]=F(S)$ (in particular, for $A[\widetilde A_1]$), the Nielsen automorphism $s\mapsto st$ does not preserve $P[\Gamma]$, since $s^2\in P[\Gamma]$ but $(st)^2\notin P[\Gamma]$. For affine types $\widetilde A_n$ and $\widetilde C_n$ the automorphism groups were determined in~\cite{ChaCri1}. Using this description one concludes that $P[\widetilde A_n]$ is characteristic in $A[\widetilde A_n]$ for $n\ge2$, and a short computation shows that $P[\widetilde C_n]$ is \emph{not characteristic} in $A[\widetilde C_n]$. 

\begin{quest}\label{q:pure-characteristic}
For which nonspherical Artin groups is the pure Artin subgroup $P[\Gamma]$ characteristic? In particular, is it characteristic for the remaining affine types $\widetilde B_n$, $\widetilde D_n$, $\widetilde E_6$, $\widetilde E_7$, $\widetilde E_8$, $\widetilde F_4$ and $\widetilde G_2$? A positive answer would give property $R_\infty$ as a consequence of Theorem~\ref{thm:1} and Lemma~\ref{lem:quotient}, since irreducible affine Coxeter groups are strongly rigid.
\end{quest}
We note that property $R_\infty$ was established for $A[\widetilde A_n]$ and $A[\widetilde C_n]$ in~\cite{CalSor1}.

\subsection{Acylindrical hyperbolicity of \texorpdfstring{$\Aut(A[\Gamma])$}{Aut(A[Gamma])}}
Large-type free-of-infinity Artin groups of rank at least $3$ are acylindrically hyperbolic~\cite{Vas22}, and their outer automorphism groups are finite~\cite{Vas25}. However, acylindrical hyperbolicity is not known in general to pass from a finite index subgroup to the ambient group, this is an open problem~\cite[Question~2]{MO19}. Thus the acylindrical hyperbolicity of the Artin group itself does not automatically imply the same for its automorphism group.

\begin{quest}\label{q:aut-artin-ah}
Is $\Aut(A[\Gamma])$ acylindrically hyperbolic for every large-type free-of-infinity Artin group of rank at least $3$? More generally, when can a known acylindrical action of an Artin group be extended to the finite group of length-preserving outer automorphisms?
\end{quest}

\subsection{Twisted Coxeter elements}
Lemma~\ref{lem:coset-infinite} constructs elements of infinite order in $W\delta$ by considering several cases. A natural uniform choice would be a \emph{twisted Coxeter element}, that is, an element $c\delta\in W\rtimes\langle\delta\rangle$, where $c$ is a product, in any order, of one simple reflection from each $\delta$-orbit in $S$. Speyer~\cite{Spe09} proved that ordinary Coxeter elements of infinite irreducible Coxeter groups are \emph{straight}, meaning that their powers are reduced. He--Nie~\cite[Proposition~3.1]{HN14} extended this result to twisted Coxeter elements of affine Coxeter groups.

\begin{quest}\label{q:twisted-speyer}
Let $(W,S)$ be an infinite irreducible Coxeter system and let $\delta$ be a Coxeter graph automorphism.
Does every twisted Coxeter element $c\delta$ have infinite order? More strongly, is it straight, i.e.\ does
$\ell\bigl(c\delta(c)\dots\delta^{n-1}(c)\bigr)=n\ell(c)$ for every $n\ge1$?
\end{quest}

\end{document}